\documentclass[10pt]{article}
\font\smallit=cmti10

\usepackage{amssymb,latexsym,amsmath,epsfig,amsthm} 
\usepackage{url,hyperref}

\makeatletter

\renewcommand\section{\@startsection {section}{1}{\z@}
{-30pt \@plus -1ex \@minus -.2ex}
{2.3ex \@plus.2ex}
{\normalfont\normalsize\bfseries\boldmath}}

\renewcommand\subsection{\@startsection{subsection}{2}{\z@}
{-3.25ex\@plus -1ex \@minus -.2ex}
{1.5ex \@plus .2ex}
{\normalfont\normalsize\bfseries\boldmath}}

\renewcommand{\@seccntformat}[1]{\csname the#1\endcsname. }

\makeatother

\newtheorem{theorem}{Theorem}
\newtheorem{lemma}{Lemma}

\newtheorem{proposition}{Proposition}
\newtheorem{corollary}{Corollary}

\newtheorem{definition}[theorem]{Definition}
\newcommand{\seqnum}[1]{\href{https://oeis.org/#1}{\underline{#1}}}

\begin{document}
\begin{center}
\uppercase{\bf Brazilian Primes Which Are Also Sophie Germain Primes}
\vskip 20pt
{\bf Jon Grantham}\\
{\smallit Institute for Defense Analyses/Center for Computing Sciences, Bowie, Maryland}\\
{\tt grantham@super.org}\\
\vskip 10pt
{\bf Hester Graves}\\
{\smallit Institute for Defense Analyses/Center for Computing Sciences, Bowie, Maryland}\\
{\tt hkgrave@super.org}\\
\end{center}
\vskip 20pt

\vskip 30pt 

\centerline{\bf Abstract}

\noindent
We disprove a conjecture of Schott that no Brazilian prime is a Sophie Germain prime. We compute all counterexamples up to $10^{46}$. We prove conditional asymptotics for the number of Brazilian Sophie Germain primes up to $x$.

\pagestyle{myheadings}
\thispagestyle{empty}
\baselineskip=12.875pt
\vskip 30pt 

\section{Introduction}

The term ``Brazilian numbers'' comes from the 1994 
Iberoamerican
Mathematical Olympiad \cite{crux} in Fortaleza, Brazil, in a problem
proposed by the Mexican math team.\footnote{The term
appears as ``sensato'' in the original problem \cite{olimp}. The authors are puzzled by the
discrepancy with \cite{crux}.}.   They became a topic of lively discussion on the forum \href{http://www.les-mathematiques.net}{mathematiques.net}.  Bernard Schott \cite{brazil} summarized the results in the standard reference on Brazilian numbers.
 
\begin{definition}
{\rm A {\it Brazilian number} $n$ is an integer whose base-$b$ representation has all the same digits for some $1<b<n-1$.} 
\end{definition}

In other words, $n$ is Brazilian if and only if $n=m\left ({\frac {b^q-1} 
{b-1}} \right )=mb^{q-1}+\cdots+mb+m$, with $q\ge 2$. These numbers are \seqnum{A125134} in the Online Encylopedia of Integer Sequences (OEIS).

A {\it Brazilian prime} (or ``prime repunit'') is a Brazilian number that is prime; by necessity, $m=1$ and $q\ge 3$. See \seqnum{A085104} in the OEIS for the sequence of Brazilian primes. In 2010, Schott \cite{brazil} conjectured that no Brazilian prime is also a Sophie Germain prime.

Sophie Germain discovered her eponymous primes while trying to prove Fermat's Last Theorem; her work was one of the first major steps towards a proof.

\begin{definition}
{\rm A {\it Sophie Germain prime} is a prime $p$ such that $2p+1$ is also prime.}
\end{definition}

Germain showed that if $p$ is such a prime, then there 
are no non-zero integers $x,y,z$, not divisible by $p$, such that $x^p + y^p = z^p$.  
If $p$ is a Sophie Germain prime, then we say that $2p+1$ is a {\it safe prime}.

It is conjectured that there are infinitely many Sophie Germain primes, but the claim is still unproven.  The Bateman-Horn conjecture \cite{BH} implies that the number of Sophie Germain primes less than $x$ is asymptotic to 
$2C \frac{x}{\log^2 x}$, where $C = \displaystyle \prod_{p>2}\frac{p(p-2)}{(p-1)^2} \approx 0.660161$.

See \cite{ribenboim} for further information about Sophie Germain primes.

\section{Enumerating Counterexamples}

To aid our search, we use a few lemmas.

\begin{lemma}\label{odd prime}
If $p={\frac {b^q-1}{b-1}}$ is a Brazilian prime, then $q$ is an odd prime.
\end{lemma}
\begin{proof}
Recall $x^{q}-1$ is divisible by the $m$th cyclotomic polynomial $\Phi_m(x)$ for $m|q$; therefore $p$ can only be prime if $q$ is also prime.  Note that
$q>2$ because $b < p-1$, so $q$ is an odd prime.
\end{proof}

The preceding lemma is also Corollary 4.1 of Schott \cite{brazil}.

\begin{lemma}
If $p$ is a Brazilian prime and a Sophie Germain prime, then $p \equiv q\equiv 2$ (mod $3$) and $b \equiv 1$ (mod $3$).
\end{lemma}
\begin{proof}
If $p$ is a Sophie Germain prime, then $3$ cannot divide the safe prime $2p+1$, so $p$ cannot be congruent to $1$ (mod $3$).  The number $3$ is not Brazilian, so $p \neq 3$ and thus $p \equiv 2$ (mod ${3}$).

If $3|b$, then 
$$p = b^{q-1} + b^{q-2} + \cdots + b + 1 \equiv 1 \pmod{3},$$
which is a contradiction.  Lemma \ref{odd prime} states that $q$ is an odd prime, so if $b \equiv 2$ (mod ${3})$, then $p \equiv 1$ (mod ${3}$), a contradiction.  We conclude that 
$b \equiv 1$ (mod ${3}$), so that $q \equiv p$ (mod ${3}$), and therefore $q \equiv 2$ (mod ${3}$).
\end{proof}

For $q=5$, we use a modification of the technique described in \cite{wolf} to compute a table of length-$5$ Brazilian primes up to $10^{46}$. We will describe this computation in full in a forthcoming paper \cite{variations}. Of these,
$104890280$ are Sophie Germain primes. The smallest is $28792661=73^4+73^3+73^2+73+1$.
We very easily prove the primality of Sophie German primes with the Pocklington-Lehmer test.

For $q\ge 11$, we very quickly enumerate all possibilities for $b\le
10^{46/(q-1)}$. 
We find $22$ Brazilian Sophie Germain primes for $q=11$, and none for
larger $q$. (We have $q<\log_2{10^{46}}+1<154$.) The smallest is $$14781835607449391161742645225951=1309^{10}+1309^9+\cdots+1309+1.$$

While we disprove Schott's conjecture, we do have a related proposition.
\begin{proposition}
The only Brazilian prime which is a safe prime is $7$.
\end{proposition}
\begin{proof}
If $p=b^{q-1}+\cdots+b+1$ is a safe prime, then $\frac{p-1}{2}=\frac 12(b^{q-1}+\cdots+b)$ must also be prime. This expression, however, is divisible by $\frac {b(b+1)}2$, which is only prime when $b=2$ and $p=7$.
\end{proof}

The list of Brazilian Sophie Germain primes is \seqnum{A306845} in the
OEIS. The
first few counterexamples were also discovered by Giovanni
Resta and Michel Marcus; see the comments for \seqnum{A085104}.

\section{Conditional Results}

The infinitude of Brazilian Sophie Germain primes, as well as the asymptotic number of them, is the consequence of well-known conjectures.

\begin{proposition}
Assuming Schnizel's Hypothesis H, there are infinitely many Brazilian Sophie Germain primes.
\end{proposition}
\begin{proof}
Recall that Hypothesis H \cite{hyph} says that any set of polynomials, whose product is not identically zero modulo any prime, is simultaneously prime infinitely often. Take our two polynomials to be $f_0(x)=x^4+x^3+x^2+x+1$ and
$f_1(x)=2x^4+2x^3+2x^2+2x+3$. Then $f_0(b)$ is Brazilian and $f_1(b)=2f_0(b)+1$. Rather than checking congruences, it suffices to note the existence of the above primes of this form to see that the conditions of Hypothesis H are satisfied.
\end{proof}

The Bateman-Horn Conjecture \cite{BH} implies a more precise statement about the number of Brazilian Sophie Germain primes.

\begin{proposition}
For an odd prime $q$, let $\Phi_q(x)$ be the $q$th cyclotomic polynomial.
Assuming the Bateman-Horn Conjecture, the number of values of $b<x$ such that $\Phi_q(b)$ and $2\Phi_q(b)+1$ are simultaneously prime is $0$ or $\sim C_q \frac{x}{{\log^2 x}}$, for some positive constant $C_q$, depending on whether
$\Phi_q(b)(2\Phi_q(b)+1)$ is identically zero modulo some prime $p$.
\end{proposition}

\begin{proof}
This follows immediately from the Bateman-Horn conjecture, with $C_q=\left(\prod_p \frac{1-N(p)/p}{(1-1/p)^2}\right)/q^2$, where $N(p)$ is the number of roots of $\Phi_q(b)(2\Phi_q(b)+1)$ modulo $p$.
\end{proof}

\begin{corollary}
Assuming the Bateman-Horn Conjecture, the number of Brazilian Sophie Germain primes up to $x$ is $\sim C \frac{x^{1/4}}{{\log^2 x}}$, for some $C$.
\end{corollary}

\begin{proof}
To find the number of Brazilian Sophie Germain primes less than $y$ of the form $\Phi_q(b)$ for a fixed $q$, we apply the preceding proposition, substituting $x=y^{1/(q-1)}$, and get $\sim C'_q \frac{y^{1/{(q-1)}}}{{\log^2 y}}$, with $C'_q=C_q(q-1)^2$. We sum over all $q\equiv 2$ (mod $3$) and notice that the $q=5$ term dominates. We can thus take $C=C'_5$.
\end{proof}

\begin{parindent}0pt
{\bf Acknowledgement. } Thanks to Enrique Trevi{\~n}o for helping track down the reference in
\cite{olimp}, and to the referee for helpful comments.
\end{parindent}

\providecommand{\bysame}{\leavevmode\hbox to3em{\hrulefill}\thinspace}
\providecommand{\MR}{\relax\ifhmode\unskip\space\fi MR }
\providecommand{\MRhref}[2]{%
  \href{http://www.ams.org/mathscinet-getitem?mr=#1}{#2}
}
\providecommand{\href}[2]{#2}

\bigskip
\hrule
\bigskip

\noindent 2020 {\it Mathematics Subject Classification}:
Primary 11A41; Secondary 11A63.

\noindent \emph{Keywords:} Brazilian number, Sophie Germain prime, prime repunit, repunit.

\end{document}